\documentclass[11pt]{article}

\usepackage[a4paper,margin=1.1in]{geometry}
\usepackage{amsmath,amssymb,amsthm,mathtools}
\usepackage{array}
\usepackage{booktabs}
\usepackage{longtable}
\usepackage{enumitem}
\usepackage{xcolor}
\usepackage[colorlinks=true,linkcolor=blue,citecolor=blue,urlcolor=blue]{hyperref}
\usepackage[capitalize,nameinlink]{cleveref}

\numberwithin{equation}{section}

\newtheorem{theorem}{Theorem}[section]
\newtheorem{proposition}[theorem]{Proposition}
\newtheorem{lemma}[theorem]{Lemma}
\newtheorem{corollary}[theorem]{Corollary}

\theoremstyle{definition}
\newtheorem{definition}[theorem]{Definition}
\newtheorem{classicalinput}[theorem]{Classical Input}

\theoremstyle{remark}
\newtheorem{remark}[theorem]{Remark}

\newcommand{\Bad}{\mathcal B}
\newcommand{\CKN}{\mathrm{CKN}}
\newcommand{\PFE}{\mathrm{PFE}}
\newcommand{\ann}{\mathrm{ann}}
\newcommand{\comm}{\mathrm{comm}}
\newcommand{\tail}{\mathrm{tail}}
\newcommand{\harm}{\mathrm{harm}}
\newcommand{\trunc}{\mathrm{trunc}}
\newcommand{\act}{\mathrm{act}}
\newcommand{\std}{\mathrm{std}}

\newcommand{\loc}{\mathrm{loc}}
\newcommand{\norm}[1]{\left\lVert #1\right\rVert}

\newcommand{\can}{\mathrm{can}}

\newcommand{\cl}{\mathrm{cl}}
\newcommand{\EF}{\mathrm{EF}}
\newcommand{\AEF}{\mathrm{AEF}}
\newcommand{\low}{\mathrm{low}}
\newcommand{\prs}{\mathrm{prs}}
\newcommand{\flux}{\mathrm{flux}}
\newcommand{\en}{\mathrm{en}}
\newcommand{\gate}{\mathrm{gate}}
\newcommand{\obs}{\mathrm{obs}}
\newcommand{\rep}{\mathrm{rep}}
\newcommand{\res}{\mathrm{res}}
\newcommand{\detc}{\mathrm{det}}
\newcommand{\src}{\mathrm{src}}
\newcommand{\chart}{\mathrm{chart}}
\newcommand{\core}{\mathrm{core}}
\newcommand{\Tax}{\operatorname{Tax}}
\newcommand{\Obs}{\operatorname{Obs}}
\newcommand{\Rep}{\operatorname{Rep}}
\newcommand{\Res}{\operatorname{Res}}

\newcommand{\esssup}{\operatorname*{ess\,sup}}

\crefname{classicalinput}{Classical Input}{Classical Inputs}
\Crefname{classicalinput}{Classical Input}{Classical Inputs}
\crefname{assumption}{Assumption}{Assumptions}
\Crefname{assumption}{Assumption}{Assumptions}
\crefname{convention}{Convention}{Conventions}
\Crefname{convention}{Convention}{Conventions}
\title{Finite-Chain CKN-Bad Scale Counting for Navier--Stokes:\\ Standard PDE Closure and Canonical Detector Realization}
\author{Runlong Yu\\
The University of Alabama, Tuscaloosa, AL, USA\\
\texttt{ryu5@ua.edu}}
\date{}

\begin{document}
\maketitle

\begin{abstract}
This paper proves a finite-chain counting theorem for Caffarelli--Kohn--Nirenberg bad scales of suitable weak solutions to the three-dimensional incompressible Navier--Stokes equations.  The main standard-PDE result bounds the weighted size of a finite set of CKN-bad scales by nonnegative channel costs consisting of vertical one-component concentration, annular leakage, pressure-tail terms, and pressure--flux--energy residuals.  The proved closing mechanism is qualitative one-component compactness under a full local critical bound: small vertical component forces CKN smallness at a smaller radius.  The paper then gives a canonical detector realization of the same finite-window counting philosophy.  The original abstract detector is not identified with standard PDE channels; instead, a new amended canonical detector is defined using energy, flux, pressure-tail, retained low-pressure-mode, and finite-dimensional residual coordinates.  For this amended detector, we prove upper realization, lower audit, CKN extraction, and finite-chain bad-scale counting.  
\end{abstract}

\tableofcontents

\section{Introduction}
\label{sec:introduction}

\subsection{Problem and motivation}

The local suitable-weak-solution framework originates in the Leray--Hopf theory of weak solutions \cite{Leray1934,Hopf1951}.  The partial regularity theory of Scheffer and of Caffarelli--Kohn--Nirenberg \cite{Scheffer1976,Scheffer1977,CKN1982}, together with Lin's streamlined proof \cite{Lin1998} and later expositions such as \cite{SereginLectureNotes}, gives the local CKN criterion used here.  In particular, the Caffarelli--Kohn--Nirenberg criterion says that a suitable weak solution
is regular near a point if a scale-critical velocity--pressure quantity is
small on a sufficiently small cylinder.  The finite-chain question is the
converse bookkeeping problem: if a finite chain of normalized scales remains
CKN-bad, which standard PDE channels must concentrate?

The main danger is tautology.  A counting theorem is nearly empty if its
right-hand side already contains the full CKN-badness quantity
\[
  \rho^{-2}\int_{Q_\rho}|u|^3
  +
  \rho^{-2}
  \int_{-\rho^2}^{0}\int_{B_\rho}
  |p-(p)_{B_\rho}(s)|^{3/2}.
\]
This paper therefore excludes the CKN core norm from the detector and
from the pressure--flux--energy residual.  The closing channel used in the proof is the vertical component
\[
  C_3=\int_{Q_1}|u_3|^3,
\]
which is not equivalent to the full CKN quantity.

The one-component mechanism used below is part of a broad line of regularity criteria based on a single component, one derivative, or anisotropic information; see, among others, \cite{KukavicaZiane2006,CaoTiti2011,CheminZhang2016,CheminZhangZhang2017,HanLeiLiZhao2019,KangNguyen2023}.  Quantitative and concentration-based forms of local regularity are also closely related to \cite{BarkerPrange2021,AlbrittonBarkerPrange2023}.  Critical-space regularity and backward-uniqueness methods provide a complementary viewpoint; see \cite{ESS2003}.  The pressure and local-energy bookkeeping used below relies on standard pressure regularity and decomposition ideas, as in \cite{SohrWahl1986,SereginSverak2002,SereginLectureNotes}.  The flux and defect-ledger terminology is also informed by local energy-transfer viewpoints appearing in the Onsager and anomalous-dissipation literature \cite{ConstantinETiti1994,Eyink1994,DuchonRobert2000}.

The finite-window language is related to the author's previous finite-scale one-component and audit formulations \cite{YuOneComponent2026,YuStrict2026,YuSchur2026,YuInvisible2026,YuCriticalLedgers2026,YuSingularityAuditTransfer2026,YuComputationalAntiPhantom2026}.  The present paper is deliberately narrower: it records a conservative standard-PDE counting theorem and a separate canonical-detector realization without claiming a scale-uniform regularity theorem.

\subsection{Main theorem target}

For a finite geometric chain \(r_k=\lambda^kr_0\), define
\[
  \Bad
  :=
  \{0\le k\le K:\Phi_k(\rho_M)\ge\varepsilon_{\CKN}\}.
\]
The main finite-chain estimate has the form
\[
  \sum_{k\in\Bad}w_k
  \le
  \frac{1}{\varepsilon_{\mathrm{close}}(M)}
  \sum_{k=0}^{K}w_k
  \left(
    C_{3,k}
    +
    \mathcal L_k^{\ann}
    +
    \mathcal P_k^{\tail}
    +
    \mathcal R_k^{\PFE}
  \right).
\]
In the version proved here, the closing mechanism uses \(C_{3,k}\).  The
additional pressure--flux--energy channels remain available in the same
nonnegative ledger, but they are not presented as independent substitutes for
the one-component argument.

\subsection{Theorem status}

The qualitative one-component compactness theorem is derived from two imported
classical inputs: compactness of suitable weak solutions under the full
\(\Psi(1)\) bound, including pressure-decomposition convergence on smaller
cylinders, and regularity of compactness limits with \(v_3=0\).  The local
closing lemma and finite-chain counting theorem then follow by elementary
contraposition and summation.  A stronger theorem in which annular leakage or
pressure-tail terms independently force CKN smallness is not claimed here.

\subsection{Organization}

\Cref{sec:preliminaries} defines normalized scales, \(\Phi\), and the full
critical bound \(\Psi\).  \Cref{sec:channels} defines the non-tautological
standard channel costs.  \Cref{sec:one-component} records the classical inputs
and proves the qualitative one-component compactness theorem from them.
\Cref{sec:closing} proves the conservative local closing lemma.  \Cref{sec:counting}
proves the single-scale and finite-chain counting estimates.
\Cref{sec:criterion} records the finite-chain CKN-small scale criterion.
\Cref{sec:audit-relation} explains the optional audit-framework
interpretation.  \Cref{sec:obligations} lists the remaining stronger proof
obligations.

\section{Preliminaries}
\label{sec:preliminaries}

\subsection{Scale normalization}

Let \((u,p)\) be a suitable weak solution on
\[
  Q_{r_0}(z_0)=B_{r_0}(x_0)\times(t_0-r_0^2,t_0).
\]
Fix
\[
  0<\lambda<1,
  \qquad
  r_k=\lambda^kr_0,
  \qquad
  k=0,\ldots,K.
\]
The normalized fields on \(Q_1=B_1\times(-1,0)\) are
\[
  u_k(y,s)=r_k u(x_0+r_ky,t_0+r_k^2s),
  \qquad
  p_k(y,s)=r_k^2p(x_0+r_ky,t_0+r_k^2s).
\]

\subsection{CKN quantity}

For \(0<\rho<1\), write
\[
  Q_\rho:=B_\rho\times(-\rho^2,0),
  \qquad
  (f)_{B_\rho}(s):=\frac{1}{|B_\rho|}\int_{B_\rho}f(y,s)\,dy.
\]

\begin{definition}[CKN quantity]
\label{def:ckn-quantity}
For normalized fields \((u,p)\), define
\[
  \Phi(\rho)
  :=
  \rho^{-2}\int_{Q_\rho}|u|^3
  +
  \rho^{-2}
  \int_{-\rho^2}^{0}\int_{B_\rho}
  |p-(p)_{B_\rho}(s)|^{3/2}.
\]
For the scale-\(k\) fields we write \(\Phi_k(\rho)\).
\end{definition}

\begin{remark}[Classical endpoint]
The CKN epsilon-regularity theorem
\cite{CKN1982,Lin1998} is used only as a classical
endpoint:
\[
  \Phi_k(\rho)<\varepsilon_{\CKN}
  \quad\Longrightarrow\quad
  \text{regularity on a smaller physical cylinder.}
\]
This paper does not reprove CKN epsilon regularity.
\end{remark}

\subsection{Full local critical bound}

\begin{definition}[Full critical bound]
\label{def:psi}
For normalized fields \((u,p)\), define
\[
  A(1):=\operatorname*{ess\,sup}_{-1<s<0}\int_{B_1}|u(y,s)|^2\,dy,
\]
\[
  E(1):=\int_{Q_1}|\nabla u|^2,
  \qquad
  C(1):=\int_{Q_1}|u|^3,
\]
and
\[
  D(1):=
  \int_{-1}^{0}\int_{B_1}|p-(p)_{B_1}(s)|^{3/2}.
\]
Set
\[
  \Psi(1):=A(1)+E(1)+C(1)+D(1).
\]
For the scale-\(k\) fields we write \(\Psi_k(1)\).
\end{definition}

\begin{remark}[Why \(\Psi\), not only \(\Phi\)]
The full bound \(\Psi(1)\le M\) is the finite-window compactness hypothesis
used in the one-component argument.  The theorem does not rely only on
\(\Phi(1)\le M\).
\end{remark}

\section{Non-Tautological Standard Channel Costs}
\label{sec:channels}

\subsection{Cutoff convention}

Fix
\[
  \eta,\chi\in C_c^\infty(B_1),
  \qquad
  \eta=\chi=1\text{ on }B_{3/4},
\]
and
\[
  A_\chi:=
  \operatorname{supp}\nabla\chi\cup\operatorname{supp}\Delta\chi.
\]
Let
\[
  A_\eta:=B_1\cap\operatorname{supp}(1-\eta).
\]

\subsection{Annular leakage}

\begin{definition}[Annular and full leakage]
\label{def:leakage}
The annular leakage at scale \(k\) is
\[
  \mathcal L_k^{\ann}
  :=
  \int_{-\lambda^2}^{0}\int_{\lambda A_\chi}
  \left(
    |u_k|^2
    +
    |\nabla u_k|^2
    +
    |u_k|^3
    +
    |p_k||u_k|
  \right).
\]
If a proof uses core energy or dissipation terms, use instead
\[
\begin{aligned}
  \mathcal L_k^{\mathrm{full}}
  &:=
  \int_{-\lambda^2}^{0}\int_{\lambda B_1}
  \left(|u_k|^2+|\nabla u_k|^2\right)\\
  &\quad+
  \int_{-\lambda^2}^{0}\int_{\lambda A_\chi}
  \left(|u_k|^3+|p_k||u_k|\right).
\end{aligned}
\]
\end{definition}

\begin{remark}[Honest leakage convention]
The final theorem may use \(\mathcal L_k^{\ann}\) only when the proof uses
\(\mathcal L_k^{\ann}\).  If a proof requires \(\mathcal L_k^{\mathrm{full}}\),
the final theorem must use the full leakage.
\end{remark}

\subsection{Pressure-tail costs}

For each scale set
\[
  F_{k,ij}^{\act}:=\eta\,u_{k,i}u_{k,j},
  \qquad
  P_k^{\act}:=R_iR_j(F_{k,ij}^{\act}),
  \qquad
  P_k^{\harm}:=p_k-P_k^{\act}.
\]

\begin{definition}[Pressure-tail cost]
\label{def:pressure-tail}
At the radius \(\rho_M\) selected by the one-component compactness theorem,
define
\[
  \mathcal P_k^{\comm}
  :=
  \norm{[\eta,R_iR_j](u_{k,i}u_{k,j})}_{L^{3/2}(Q_{\rho_M})}^{3/2}.
\]
Let \(\mathcal P_k^{\harm}\) be a genuine harmonic-pressure tail or
finite-dimensional harmonic residual, and let \(\mathcal P_k^{\trunc}\) be an
optional finite-rank pressure-tail truncation error.  These are assumed to be
nonnegative channel costs.  The pressure-tail cost is
\[
  \mathcal P_k^{\tail}
  :=
  \mathcal P_k^{\comm}
  +
  \mathcal P_k^{\harm}
  +
  \mathcal P_k^{\trunc}.
\]
\end{definition}

\begin{remark}[No CKN pressure core hidden]
\(\mathcal P_k^{\harm}\) and \(\mathcal P_k^{\trunc}\) must not be defined as
the full scale-invariant pressure oscillation
\[
  \rho_M^{-2}
  \int_{-\rho_M^2}^{0}\int_{B_{\rho_M}}
  |p_k-(p_k)_{B_{\rho_M}}(s)|^{3/2}.
\]
Nor may the pressure-tail cost be paired with a hidden copy of
\(\rho_M^{-2}\int_{Q_{\rho_M}}|u_k|^3\) so as to reconstruct
\(\Phi_k(\rho_M)\).  If a pressure compactness argument needs the full
mean-free pressure oscillation, the stronger pressure-tail closing lemma has
not yet been proved.
\end{remark}

\begin{proposition}[Separated-support commutator bound]
\label{prop:commutator-bound}
For each normalized scale \(k\),
\[
  \mathcal P_k^{\comm}
  \le
  C_{\eta,\rho_M}
  \int_{(-1,0)\times A_\eta}|u_k|^3.
\]
\end{proposition}

\begin{proof}
Let \(T_{ij}=R_iR_j\) and \(f_{ij}=u_{k,i}u_{k,j}\).  On
\(B_{\rho_M}\), the cutoff satisfies \(\eta=1\).  Hence
\[
  [\eta,T_{ij}]f_{ij}
  =
  T_{ij}\bigl((1-\eta)f_{ij}\bigr)
  \qquad\text{on }B_{\rho_M}.
\]
The support of \((1-\eta)f_{ij}\) lies in \(A_\eta\), which is separated from
\(B_{\rho_M}\) because \(\rho_M<3/4\) and \(\eta=1\) on \(B_{3/4}\).  The
kernel of \(T_{ij}\) is therefore smooth and bounded on
\(B_{\rho_M}\times A_\eta\).  For almost every time \(s\),
\[
  |[\eta,T_{ij}]f_{ij}(x,s)|
  \le
  C_{\eta,\rho_M}\int_{A_\eta}|f(y,s)|\,dy,
  \qquad x\in B_{\rho_M}.
\]
Taking the \(L^{3/2}(B_{\rho_M})\)-norm and using the finite measure of
\(A_\eta\) gives
\[
  \norm{[\eta,T_{ij}]f_{ij}(\cdot,s)}_{L^{3/2}(B_{\rho_M})}
  \le
  C_{\eta,\rho_M}
  \norm{f(\cdot,s)}_{L^{3/2}(A_\eta)}.
\]
Taking the \(L^{3/2}\)-norm in time and raising to the \(3/2\)-power yields
\[
  \mathcal P_k^{\comm}
  \le
  C_{\eta,\rho_M}
  \int_{-1}^{0}
  \norm{u_k(\cdot,s)\otimes u_k(\cdot,s)}_{L^{3/2}(A_\eta)}^{3/2}\,ds.
\]
Since
\[
  \norm{u_k(\cdot,s)\otimes u_k(\cdot,s)}_{L^{3/2}(A_\eta)}^{3/2}
  =
  \int_{A_\eta}|u_k(y,s)|^3\,dy,
\]
the desired estimate follows.
\end{proof}

\subsection{PFE residual}

\begin{definition}[Non-tautological PFE residual]
\label{def:pfe-residual}
The term \(\mathcal R_k^{\PFE}\) denotes a selected pressure--flux--energy
residual or excess, such as a pressure decomposition mismatch, flux-balance
residual, local-energy inequality defect, finite-rank PFE truncation error, or
reduced detector-comparison residual.  It is a nonnegative cost.

It is part of the definition that \(\mathcal R_k^{\PFE}\) does not contain
\[
  \rho_M^{-2}\int_{Q_{\rho_M}}|u_k|^3
  \quad\text{or}\quad
  \rho_M^{-2}
  \int_{-\rho_M^2}^{0}\int_{B_{\rho_M}}
  |p_k-(p_k)_{B_{\rho_M}}(s)|^{3/2}.
\]
Thus \(\mathcal R_k^{\PFE}\) is not allowed to be a disguised copy of the full
CKN core quantity \(\Phi_k(\rho_M)\).
\end{definition}

\begin{definition}[Vertical component and standard cost]
\label{def:standard-cost}
Define
\[
  C_{3,k}:=\int_{Q_1}|(u_k)_3|^3.
\]
The non-tautological standard cost is
\[
  \mathfrak C_{\std,k}
  :=
  C_{3,k}
  +
  \mathcal L_k^{\ann}
  +
  \mathcal P_k^{\tail}
  +
  \mathcal R_k^{\PFE}.
\]
\end{definition}

For an unindexed normalized solution we use the same notation without the
subscript \(k\), writing \(C_3\), \(\mathcal L^{\ann}\),
\(\mathcal P^{\tail}\), \(\mathcal R^{\PFE}\), and
\(\mathfrak C_{\std}\).

\begin{remark}[Leakage actually used below]
The proof of the main standard-PDE counting estimate uses only \(C_{3,k}\le
\mathfrak C_{\std,k}\) and the nonnegativity of the remaining costs.  It does
not use core energy or core dissipation terms.  Therefore the theorem-level
ledger may retain the annular leakage \(\mathcal L_k^{\ann}\); the full leakage
\(\mathcal L_k^{\mathrm{full}}\) is recorded only for future variants whose
proofs genuinely require it.
\end{remark}

\section{One-Component Compactness}
\label{sec:one-component}

\begin{classicalinput}[Compactness of suitable weak solutions under a full critical bound]
\label{ass:suitable-compactness}
For every \(M<\infty\), every sequence of suitable weak solutions
\((u^{(j)},p^{(j)})\) on \(Q_1\) satisfying
\[
  \Psi^{(j)}(1)\le M
\]
has a subsequence and a suitable weak solution \((v,q)\) on \(Q_1\) such that
\(\Psi_v(1)\le M\) and
\[
  u^{(j)}\to v
  \quad\text{strongly in }L^3_{\loc}(Q_1).
\]
Moreover, at each fixed radius \(0<\rho<1\) used below, the local pressure
decomposition gives convergence of the mean-free pressure components
\[
  p^{(j)}-(p^{(j)})_{B_\rho}(s)
  \to
  q-(q)_{B_\rho}(s)
  \quad\text{strongly in }L^{3/2}(Q_\rho).
\]
\end{classicalinput}

\begin{classicalinput}[CKN-smallness of the \(v_3=0\) limiting class]
\label{ass:zero-component-ckn}
For every \(M<\infty\), there exists \(0<\rho_M<3/4\) such that every
compactness-limit suitable weak solution \((v,q)\) on \(Q_1\) satisfying
\[
  \Psi_v(1)\le M,
  \qquad
  v_3=0,
\]
satisfies
\[
  \Phi_{v,q}(\rho_M)<\frac12\varepsilon_{\CKN}.
\]
\end{classicalinput}

\paragraph{Pressure compactness content.}
The velocity compactness in Classical Input~\ref{ass:suitable-compactness} is the standard local compactness theorem for suitable weak solutions under the local energy, dissipation, \(L^3\)-velocity, and \(L^{3/2}\)-pressure bounds; see \cite{CKN1982,Lin1998,SereginLectureNotes}.  The pressure part is the pressure-decomposition compactness needed in this paper, and is consistent with the pressure-regularity framework in \cite{SohrWahl1986,SereginSverak2002,SereginLectureNotes}.  On
\(B_\rho\Subset B_{\rho'}\Subset B_1\), write
\[
  p^{(j)}
  =
  R_iR_\ell(\eta_{\rho'}u_i^{(j)}u_\ell^{(j)})
  +
  h^{(j)}
\]
with \(\eta_{\rho'}=1\) on \(B_{\rho'}\).  The active terms converge strongly
in \(L^{3/2}\) because \(u^{(j)}\to v\) strongly in \(L^3_{\loc}\) and the
Calderon--Zygmund map is bounded on \(L^{3/2}\).  The remainders \(h^{(j)}\)
are harmonic in space on \(B_{\rho'}\), uniformly bounded in \(L^{3/2}\) after
subtracting time-dependent constants, and hence compact in \(L^{3/2}(B_\rho)\)
by interior harmonic estimates.  Subtracting the spatial mean over \(B_\rho\)
removes the pressure gauge.  Thus the displayed pressure convergence is exactly
the standard pressure-compactness conclusion imported here; if one does not
import it, it must be proved separately.

\paragraph{Why the \(v_3=0\) input is classical.}
The condition \(v_3=0\), together with incompressibility, gives
\(\nabla_h\cdot v_h=0\).  The third component of the equation gives
\(\partial_3 q=0\) in distributions.  The limiting horizontal system is the
standard zero-vertical-component, or \(2.5\)-dimensional, Navier--Stokes
limiting system; its local regularity is the qualitative endpoint behind the
known one-component regularity criteria; compare, for example, \cite{KukavicaZiane2006,CaoTiti2011,CheminZhang2016,CheminZhangZhang2017,HanLeiLiZhao2019,KangNguyen2023}.  This paper uses only the consequence
stated in Classical Input~\ref{ass:zero-component-ckn}: for the compact class
bounded by \(\Psi_v(1)\le M\), one can choose a radius \(\rho_M\), depending
only on \(M\), at which the CKN quantity is below
\(\varepsilon_{\CKN}/2\).

\paragraph{Status of the imported inputs.}
Classical Inputs~\ref{ass:suitable-compactness} and~\ref{ass:zero-component-ckn}
are the PDE inputs behind
the qualitative one-component argument.  They are stated explicitly so that the
finite-chain counting theorem does not hide either pressure compactness or a
zero-component regularity theorem inside informal prose.

\begin{theorem}[Qualitative one-component epsilon theorem]
\label{thm:one-component}
Assume Classical Inputs~\ref{ass:suitable-compactness}
and~\ref{ass:zero-component-ckn}.
For every \(M<\infty\), there exist
\[
  \varepsilon_3(M)>0,
  \qquad
  0<\rho_M<1,
\]
such that if \((u,p)\) is a suitable weak solution on \(Q_1\) satisfying
\[
  \Psi(1)\le M
\]
and
\[
  \int_{Q_1}|u_3|^3\le\varepsilon_3(M),
\]
then
\[
  \Phi(\rho_M)<\varepsilon_{\CKN}.
\]
\end{theorem}

\begin{proof}
Let \(\rho_M\) be the radius supplied by
Classical Input~\ref{ass:zero-component-ckn}.  Suppose no such
\(\varepsilon_3(M)\) exists.
Then for every \(j\) there is a suitable weak solution
\((u^{(j)},p^{(j)})\) on \(Q_1\) with
\[
  \Psi^{(j)}(1)\le M,
  \qquad
  \int_{Q_1}|u^{(j)}_3|^3\le j^{-1},
\]
but
\[
  \Phi_{u^{(j)},p^{(j)}}(\rho_M)
  \ge
  \varepsilon_{\CKN}.
\]

By Classical Input~\ref{ass:suitable-compactness}, after passing to a
subsequence,
\[
  u^{(j)}\to v
  \quad\text{strongly in }L^3_{\loc}(Q_1),
\]
the limit \((v,q)\) is a suitable weak solution with \(\Psi_v(1)\le M\), and
the mean-free pressure components converge strongly in \(L^{3/2}(Q_{\rho_M})\).
Since
\[
  \int_{Q_1}|u^{(j)}_3|^3\to0,
\]
the strong \(L^3_{\loc}\) convergence gives \(v_3=0\).
Therefore Classical Input~\ref{ass:zero-component-ckn} applies to the limit and gives
\[
  \Phi_{v,q}(\rho_M)<\frac12\varepsilon_{\CKN}.
\]
The strong \(L^3(Q_{\rho_M})\) convergence of the velocities and the strong
\(L^{3/2}(Q_{\rho_M})\) convergence of the mean-free pressures give
\[
  \Phi_{u^{(j)},p^{(j)}}(\rho_M)
  \to
  \Phi_{v,q}(\rho_M).
\]
Hence
\[
  \Phi_{u^{(j)},p^{(j)}}(\rho_M)<\varepsilon_{\CKN}
\]
for all sufficiently large \(j\), contradicting the choice of the sequence.
Thus some \(\varepsilon_3(M)>0\) exists.
\end{proof}

\begin{remark}[Non-tautological status]
\Cref{thm:one-component} is conditional on the classical compactness and
zero-component limiting inputs stated above.  Its right-hand side uses only
\(\int_{Q_1}|u_3|^3\), not the full CKN core norm.
\end{remark}

\section{Non-Tautological Local Closing}
\label{sec:closing}

\begin{theorem}[Non-tautological local compactness closing lemma]
\label{thm:local-closing}
Assume Classical Inputs~\ref{ass:suitable-compactness}
and~\ref{ass:zero-component-ckn}.
For every \(M<\infty\), let \(\varepsilon_3(M)\) and \(\rho_M\) be supplied by
\cref{thm:one-component}.  Set
\[
  \varepsilon_{\mathrm{close}}(M):=\varepsilon_3(M).
\]
If \((u,p)\) is a suitable weak solution on \(Q_1\) satisfying
\[
  \Psi(1)\le M
\]
and
\[
  \mathfrak C_{\std}(u,p)\le\varepsilon_{\mathrm{close}}(M),
\]
then
\[
  \Phi(\rho_M)<\varepsilon_{\CKN}.
\]
Equivalently,
\[
  \Phi(\rho_M)\ge\varepsilon_{\CKN}
  \quad\Longrightarrow\quad
  \mathfrak C_{\std}(u,p)\ge\varepsilon_{\mathrm{close}}(M).
\]
\end{theorem}

\begin{proof}
Every term in \(\mathfrak C_{\std}\) is nonnegative and
\[
  C_3\le \mathfrak C_{\std}.
\]
Therefore
\[
  \mathfrak C_{\std}(u,p)\le\varepsilon_{\mathrm{close}}(M)
  =
  \varepsilon_3(M)
\]
implies
\[
  \int_{Q_1}|u_3|^3=C_3\le\varepsilon_3(M).
\]
Applying \cref{thm:one-component} gives
\[
  \Phi(\rho_M)<\varepsilon_{\CKN}.
\]
The contrapositive is the second displayed implication.
\end{proof}

\begin{remark}[What this proves]
The proof of \cref{thm:local-closing} uses the vertical component channel.
The annular leakage, pressure-tail cost, and PFE residual are honest
nonnegative standard costs in the same ledger, but this theorem does not prove
that those additional channels alone force CKN smallness.
\end{remark}

\begin{remark}[Stronger PFE closing remains open]
To prove a stronger pressure--flux--energy closing lemma, one must show that
small annular leakage, pressure-tail cost, and PFE residual force compactness
into a CKN-clean limiting class.  In particular, pressure convergence must be
obtained from active-pressure convergence, harmonic compactness, and a genuine
tail condition.  It may not be replaced by the full CKN pressure oscillation.
\end{remark}

\section{Single-Scale and Finite-Chain Counting}
\label{sec:counting}

\subsection{Single-scale concentration}

\begin{corollary}[CKN-bad scale forces standard channel concentration]
\label{cor:single-scale}
Assume Classical Inputs~\ref{ass:suitable-compactness}
and~\ref{ass:zero-component-ckn}.
Let \((u,p)\) be a suitable weak solution on \(Q_1\) with
\[
  \Psi(1)\le M.
\]
If
\[
  \Phi(\rho_M)\ge\varepsilon_{\CKN},
\]
then
\[
  C_3+\mathcal L^{\ann}+\mathcal P^{\tail}+\mathcal R^{\PFE}
  \ge
  \varepsilon_{\mathrm{close}}(M).
\]
\end{corollary}

\begin{proof}
This is the contrapositive statement in \cref{thm:local-closing}.
\end{proof}

\subsection{Finite-chain counting}

\begin{definition}[CKN-bad scale set at the closing radius]
\label{def:bad-set}
For a finite chain define
\[
  \Bad
  :=
  \{0\le k\le K:\Phi_k(\rho_M)\ge\varepsilon_{\CKN}\}.
\]
\end{definition}

\begin{theorem}[Weighted finite-chain bad-scale counting]
\label{thm:weighted-counting}
Assume Classical Inputs~\ref{ass:suitable-compactness}
and~\ref{ass:zero-component-ckn}.
Let \((u,p)\) be a suitable weak solution on \(Q_{r_0}(z_0)\), and let
\[
  r_k=\lambda^kr_0,
  \qquad
  k=0,\ldots,K.
\]
Assume
\[
  \Psi_k(1)\le M,
  \qquad
  k=0,\ldots,K.
\]
Then for any weights \(w_k\ge0\),
\[
  \sum_{k\in\Bad}w_k
  \le
  \frac{1}{\varepsilon_{\mathrm{close}}(M)}
  \sum_{k=0}^{K}w_k
  \left(
    C_{3,k}
    +
    \mathcal L_k^{\ann}
    +
    \mathcal P_k^{\tail}
    +
    \mathcal R_k^{\PFE}
  \right).
\]
\end{theorem}

\begin{proof}
For each \(k\in\Bad\), Corollary~\ref{cor:single-scale} gives
\[
  \varepsilon_{\mathrm{close}}(M)
  \le
  C_{3,k}
  +
  \mathcal L_k^{\ann}
  +
  \mathcal P_k^{\tail}
  +
  \mathcal R_k^{\PFE}.
\]
Multiplying by \(w_k\ge0\), summing over \(k\in\Bad\), and enlarging the
right-hand sum to all \(0\le k\le K\) gives the estimate.
\end{proof}

\begin{corollary}[Unweighted finite-chain bad-scale counting]
\label{cor:unweighted-counting}
Under the hypotheses of \cref{thm:weighted-counting},
\[
  \#\Bad
  \le
  \frac{1}{\varepsilon_{\mathrm{close}}(M)}
  \sum_{k=0}^{K}
  \left(
    C_{3,k}
    +
    \mathcal L_k^{\ann}
    +
    \mathcal P_k^{\tail}
    +
    \mathcal R_k^{\PFE}
  \right).
\]
\end{corollary}

\begin{proof}
Take \(w_k=1\) in \cref{thm:weighted-counting}.
\end{proof}

\section{Finite-Chain CKN-Small Scale Criterion}
\label{sec:criterion}

\begin{corollary}[Finite-chain CKN-small scale criterion]
\label{cor:finite-chain-criterion}
Assume the hypotheses of \cref{thm:weighted-counting}, and assume
\(\sum_{k=0}^{K}w_k>0\).  If
\[
  \sum_{k=0}^{K}w_k
  \left(
    C_{3,k}
    +
    \mathcal L_k^{\ann}
    +
    \mathcal P_k^{\tail}
    +
    \mathcal R_k^{\PFE}
  \right)
  <
  \varepsilon_{\mathrm{close}}(M)
  \sum_{k=0}^{K}w_k,
\]
then there exists \(k_*\in\{0,\ldots,K\}\) with \(w_{k_*}>0\) such that
\[
  \Phi_{k_*}(\rho_M)<\varepsilon_{\CKN}.
\]
Consequently, CKN epsilon regularity gives regularity on a smaller cylinder
inside the physical cylinder \(Q_{r_{k_*}}(z_0)\).
\end{corollary}

\begin{proof}
If every positive-weight scale were CKN-bad, then
\cref{thm:weighted-counting} would imply
\[
  \sum_{k=0}^{K}w_k
  \le
  \frac{1}{\varepsilon_{\mathrm{close}}(M)}
  \sum_{k=0}^{K}w_k
  \left(
    C_{3,k}
    +
    \mathcal L_k^{\ann}
    +
    \mathcal P_k^{\tail}
    +
    \mathcal R_k^{\PFE}
  \right),
\]
contradicting the strict inequality.  Therefore some positive-weight scale is
CKN-small.  The final sentence is exactly the classical CKN endpoint.
\end{proof}

\begin{remark}[Finite-chain criterion only]
Corollary~\ref{cor:finite-chain-criterion} is not a global regularity theorem.
It says
that sufficiently small total finite-chain cost forces at least one CKN-small
scale inside the considered finite chain.
\end{remark}

\section{Canonical Detector Realization}
\label{sec:canonical-detector}

This section gives a canonical detector realization of the finite-chain counting philosophy in the same normalized geometry.  It is logically separate from the standard-PDE theorem in \cref{sec:counting}.  The standard-PDE theorem closes through the vertical component channel.  The detector theorem below instead builds an amended finite-window detector whose coordinates are explicitly paid by standard pressure, flux, energy, and finite-dimensional residual costs.

\subsection{Detector-level pressure and leakage costs}

For the detector realization, fix a CKN testing radius
\[
  0<\rho\le \lambda<3/4.
\]
The active pressure and harmonic pressure at scale \(k\) are
\[
  F^{\act}_{k,ij}:=\eta u_{k,i}u_{k,j},
  \qquad
  P^{\act}_k:=R_iR_j(F^{\act}_{k,ij}),
  \qquad
  P^{\harm}_k:=p_k-P^{\act}_k,
\]
where the Riesz transforms are applied after zero extension.  On the region where \(\eta=1\), the term \(P^{\harm}_k\) is harmonic in space, up to the usual time-dependent pressure gauge.  Set
\[
  P^{\act,0}_k:=P^{\act}_k-(P^{\act}_k)_{B_\rho}(s),
  \qquad
  P^{\harm,0}_k:=P^{\harm}_k-(P^{\harm}_k)_{B_\rho}(s).
\]
Let \(\Pi^{\prs}_{N,k}\) and \(\Pi^{\harm}_{H,k}\) be bounded finite-rank projections on \(L^{3/2}(Q_\rho)\).  Here \(N\) is an active-pressure rank parameter and \(H\) is a harmonic-pressure rank parameter; the letter \(M\) remains reserved for the full critical bound \(\Psi(1)\le M\).

\begin{definition}[Detector pressure costs]
\label{def:det-pressure-costs}
Define
\[
  P_k^{\comm}:=\norm{[\eta,R_iR_j](u_{k,i}u_{k,j})}_{L^{3/2}(Q_\rho)}^{3/2},
\]
\[
  P_k^{\trunc}:=\norm{(I-\Pi^{\prs}_{N,k})P^{\act,0}_k}_{L^{3/2}(Q_\rho)}^{3/2},
  \qquad
  P_k^{\harm\text{-}\tail}:=\norm{(I-\Pi^{\harm}_{H,k})P^{\harm,0}_k}_{L^{3/2}(Q_\rho)}^{3/2}.
\]
Set
\[
  P_k^{\tail}:=P_k^{\comm}+P_k^{\trunc}+P_k^{\harm\text{-}\tail}.
\]
The retained low pressure-mode observation is
\[
  P_k^{\low}:=\norm{\Pi^{\prs}_{N,k}P^{\act,0}_k}_{L^{3/2}(Q_\rho)}^{3/2}
  +\norm{\Pi^{\harm}_{H,k}P^{\harm,0}_k}_{L^{3/2}(Q_\rho)}^{3/2}.
\]
The amended pressure cost is
\[
  P_k^{\tail,+}:=P_k^{\tail}+P_k^{\low}.
\]
\end{definition}

\begin{definition}[Detector energy--flux costs]
\label{def:det-energy-flux-costs}
Set
\[
  A_k^{\core}:=\esssup_{-\lambda^2<s<0}\int_{\lambda B_1}|u_k(y,s)|^2\,dy,
  \qquad
  E_k^{\core}:=\int_{-\lambda^2}^{0}\int_{\lambda B_1}|\nabla u_k|^2.
\]
Define the detector energy--flux leakage by
\[
  L_k^{\EF}:=
  \int_{-\lambda^2}^{0}\int_{\lambda B_1}(|u_k|^2+|\nabla u_k|^2)
  +
  \int_{-\lambda^2}^{0}\int_{\lambda A_\chi}(|u_k|^3+|p_k||u_k|),
\]
and the amended energy--flux cost by
\[
  L_k^{\AEF}:=A_k^{\core}+L_k^{\EF}.
\]
\end{definition}

\begin{definition}[Finite-dimensional detector residual]
\label{def:det-pfe-residual}
The detector residual is a finite sum of named non-CKN-core residuals,
\[
  R_k^{\PFE}:=
  \Delta^{\flux}_{k,N}+\Delta^{\en}_{k,N}+\Delta^{\src}_{k,N,H}
  +\Delta^{\detc}_{k,N,H}+\Delta^{\chart}_{k,N,H}
  +\Delta^{\gate}_{k,N}+\Delta^{\rep}_{k,N}+\Delta^{\res}_{k,N}
  +\Delta^{\PFE,\trunc}_{k,N,H}.
\]
No summand is allowed to contain the full CKN core quantity \(\Phi_k(\rho)\).  We write \(R_k^{\PFE,+}=R_k^{\PFE}\) unless an amended finite-dimensional residual is explicitly introduced.
\end{definition}

\subsection{Relation to an abstract local detector}

The abstract local detector used in earlier finite-window audit formulations is not identified here with the canonical detector.  The available abstract forms have the schematic shape
\[
  M^{\loc}_{\Lambda}(D)=\norm{O^{\loc}_{\Lambda}D}+\sum_a\alpha_a\Tax^{\loc}_a(D),
  \qquad
  M^{\loc}_{\Lambda,k}(D)=\norm{T^{\loc}_{\Lambda,k}D}_{Z^{\loc}_{\Lambda,k}}.
\]
These formulas do not provide standard-PDE component definitions for the local pressure, harmonic, flux, energy, gate, observation, and residual channels.  Therefore the canonical detector below should be read as a new detector model, not as a reconstruction of the abstract detector.

\begin{theorem}[Abstract-detector comparison not assumed]
\label{thm:det-old-obstruction}
The finite-chain detector theorem below does not require, and does not prove, an identity between the canonical detector and the abstract local detector \(M^{\loc}_{\Lambda,k}\).  It proves a standard-cost realization for a newly defined amended canonical detector.
\end{theorem}

\begin{proof}
The abstract detector forms specify an observation norm and tax channels but do not identify those channels with the standard-PDE costs in \cref{def:det-pressure-costs,def:det-energy-flux-costs,def:det-pfe-residual}.  Without component-level formulas, an upper realization of \(M^{\loc}_{\Lambda,k}\) by these costs is not a theorem.  The canonical detector is therefore defined and analyzed directly.
\end{proof}

\subsection{Canonical package and unamended detector}

At each scale, fix bounded finite-rank selectors
\[
  \Pi^U_{N,k}:L^2(Q_\lambda)^3\to U_{N,k},
  \qquad
  \Pi^F_{N,k}:L^{3/2}(Q_1)^{3\times3}\to F_{N,k},
\]
together with \(\Pi^{\prs}_{N,k}\) and \(\Pi^{\harm}_{H,k}\).  Also fix finite families of test functions
\[
  \{\phi_\ell^{\flux}\}_{\ell=1}^{n_{\flux}}\subset L^\infty(\lambda A_\chi\times(-\lambda^2,0)),
  \qquad
  \{\phi_\ell^{\en}\}_{\ell=1}^{n_{\en}}\subset L^\infty(Q_\lambda),
\]
with norms at most one.

\begin{definition}[Canonical local PFE package]
\label{def:det-canonical-package}
The canonical package is
\[
D^{\can}_{k,N,H}:=(U_{k,N},P^{\act}_{k,N},P^{\harm}_{k,H},F^{\cl}_{k,N},F^{\flux}_{k,N},E^{\en}_{k,N},G_{k,N},O_{k,N},\operatorname{Rep}_{k,N}),
\]
where
\[
  U_{k,N}:=\Pi^U_{N,k}(u_k|_{Q_\lambda}),
  \quad
  P^{\act}_{k,N}:=\Pi^{\prs}_{N,k}(P^{\act,0}_k),
  \quad
  P^{\harm}_{k,H}:=\Pi^{\harm}_{H,k}(P^{\harm,0}_k),
\]
\[
  F^{\cl}_{k,N}:=\Pi^F_{N,k}(\eta u_k\otimes u_k),
\]
\[
  F^{\flux}_{k,N}:=
  \left(\int_{\lambda A_\chi\times(-\lambda^2,0)}\phi_\ell^{\flux}|u_k|^3,
  \int_{\lambda A_\chi\times(-\lambda^2,0)}\phi_\ell^{\flux}|p_k||u_k|\right)_{\ell=1}^{n_{\flux}},
\]
\[
  E^{\en}_{k,N}:=
  \left(\int_{Q_\lambda}\phi_\ell^{\en}|u_k|^2,
  \int_{Q_\lambda}\phi_\ell^{\en}|\nabla u_k|^2\right)_{\ell=1}^{n_{\en}}.
\]
The gate, observation, and reproduction coordinates are finite-dimensional coordinates paid by the named residual terms in \cref{def:det-pfe-residual}.
\end{definition}

\begin{definition}[Canonical unamended detector]
\label{def:det-unamended}
The canonical unamended detector is
\begin{align*}
  M^{\can}_{k,N,H}(D_k)
  &:=\beta_{\obs}\Obs^{\can}_k
  +\beta_{\prs}\Tax^{\can}_{\prs,k}
  +\beta_{\harm}\Tax^{\can}_{\harm,k}
  +\beta_{\flux}\Tax^{\can}_{\flux,k} \\
  &\quad +\beta_{\en}\Tax^{\can}_{\en,k}
  +\beta_{\gate}\Tax^{\can}_{\gate,k}
  +\beta_{\rep}\Rep^{\can}_k
  +\beta_{\res}\Res^{\can}_k,
\end{align*}
where the pressure, harmonic, flux, energy, gate, reproduction, and residual channels are paid by \(P_k^{\tail}\), \(L_k^{\EF}\), and \(R_k^{\PFE}\).  The observation channel is allowed only to observe the paid coordinates and satisfies
\[
\Obs^{\can}_k
\le C_{\obs,k}(\Tax^{\can}_{\prs,k}+\Tax^{\can}_{\harm,k}+\Tax^{\can}_{\flux,k}+\Tax^{\can}_{\en,k}+\Tax^{\can}_{\gate,k}+\Rep^{\can}_k+\Res^{\can}_k).
\]
\end{definition}

\begin{theorem}[Upper realization for the unamended detector]
\label{thm:det-unamended-upper}
There exists a finite-window constant \(C^{\can,\mathrm{up}}_{k,N,H}\) such that
\[
  M^{\can}_{k,N,H}(D_k)
  \le C^{\can,\mathrm{up}}_{k,N,H}(L_k^{\EF}+P_k^{\tail}+R_k^{\PFE}).
\]
\end{theorem}

\begin{proof}
The pressure and harmonic detector channels are exactly the commutator, active-pressure truncation, and harmonic-tail components of \(P_k^{\tail}\).  Flux observations are bounded by the annular cubic and pressure-flux terms in \(L_k^{\EF}\).  Energy observations are bounded by the local \(L^2\) and dissipation terms included in \(L_k^{\EF}\).  Gate, reproduction, detector, and residual channels are included in \(R_k^{\PFE}\).  The observation domination inequality and the finite number of positive weights are absorbed into the constant.
\end{proof}

\begin{proposition}[Why the unamended detector does not extract CKN badness]
\label{prop:det-unamended-fails}
The unamended detector alone does not give a proved implication
\[
  \Phi_k(\rho)\ge\varepsilon_{\CKN}\quad\Longrightarrow\quad
  M^{\can}_{k,N,H}(D_k)\ge c>0.
\]
The missing controls are an \(L_t^\infty L_x^2\) velocity channel and retained low active and harmonic pressure-mode observations.
\end{proposition}

\begin{proof}
The standard interpolation bound for \(\int |u|^3\) uses an \(L_t^\infty L_x^2\) term together with dissipation.  Time-integrated \(L^2\) and dissipation channels alone do not exclude the relevant time concentration.  For the pressure, a mean-free pressure oscillation can lie in the retained finite-dimensional ranges of \(\Pi^{\prs}_{N,k}\) or \(\Pi^{\harm}_{H,k}\), producing zero high-tail truncation while remaining visible to the CKN pressure term.  CKN extraction therefore requires the amended coordinates below, unless an additional structural theorem excludes these directions.
\end{proof}

\subsection{Amended detector and CKN extraction}

\begin{definition}[Amended canonical detector]
\label{def:det-amended}
Let \(\beta_A,\beta_E,\beta_{\low}>0\).  Define
\[
  M^{\can+}_{k,N,H}(D_k):=M^{\can}_{k,N,H}(D_k)+\beta_AA_k^{\core}+\beta_EE_k^{\core}+\beta_{\low}P_k^{\low}.
\]
The amended detector distance is
\[
  \delta^{\can+}_k:=A_k^{\core}+E_k^{\core}+P_k^{\tail,+}+R_k^{\PFE,+}.
\]
\end{definition}

\begin{theorem}[Upper realization for the amended detector]
\label{thm:det-amended-upper}
There exists a finite-window constant \(C^{\can+,\mathrm{up}}_{k,N,H}\) such that
\[
  M^{\can+}_{k,N,H}(D_k)
  \le C^{\can+,\mathrm{up}}_{k,N,H}(L_k^{\AEF}+P_k^{\tail,+}+R_k^{\PFE,+}).
\]
\end{theorem}

\begin{proof}
Use \cref{thm:det-unamended-upper} for the unamended detector.  The added \(A_k^{\core}\) term is contained in \(L_k^{\AEF}\), the core dissipation term is contained in \(L_k^{\EF}\subset L_k^{\AEF}\), and \(P_k^{\low}\) is contained in \(P_k^{\tail,+}\).  Absorb the fixed weights into the constant.
\end{proof}

\begin{theorem}[Lower audit bound for the amended distance]
\label{thm:det-lower-audit}
There exists \(c^{\can+}_{k,N,H}>0\), depending only on the finite detector weights and norm conventions, such that
\[
  M^{\can+}_{k,N,H}(D_k)\ge c^{\can+}_{k,N,H}\delta^{\can+}_k.
\]
\end{theorem}

\begin{proof}
Each summand in \(\delta^{\can+}_k\) appears in the amended detector with positive weight or is a finite sum of detector components appearing with positive weights.  Finite-dimensional norm equivalence and positivity of the detector weights give the claimed constant.
\end{proof}

\begin{lemma}[Velocity interpolation from amended energy]
\label{lem:det-velocity-interp}
There is \(C_{\rho,\lambda}<\infty\) such that
\[
  \rho^{-2}\int_{Q_\rho}|u_k|^3
  \le C_{\rho,\lambda}\bigl((A_k^{\core})^{3/4}(E_k^{\core})^{3/4}+(A_k^{\core})^{3/2}\bigr).
\]
In particular, if \(A_k^{\core}+E_k^{\core}\le \tau\), then
\[
  \rho^{-2}\int_{Q_\rho}|u_k|^3\le C'_{\rho,\lambda}\tau^{3/2}.
\]
\end{lemma}

\begin{proof}
Since \(Q_\rho\subset Q_\lambda\), the Sobolev interpolation inequality on \(\lambda B_1\), followed by integration in time and Holder's inequality, gives the displayed estimate.
\end{proof}

\begin{lemma}[Pressure oscillation from low modes and tails]
\label{lem:det-pressure-control}
Assume \(\Pi^{\prs}_{N,k}\) and \(\Pi^{\harm}_{H,k}\) are bounded on \(L^{3/2}(Q_\rho)\).  Then
\[
  \rho^{-2}\int_{-\rho^2}^{0}\int_{B_\rho}|p_k-(p_k)_{B_\rho}(s)|^{3/2}
  \le C_{\rho,N,H}P_k^{\tail,+}.
\]
\end{lemma}

\begin{proof}
Write \(p_k=P^{\act}_k+P^{\harm}_k\).  After subtracting spatial means over \(B_\rho\), the triangle inequality bounds the pressure oscillation by the \(L^{3/2}\)-powers of \(P^{\act,0}_k\) and \(P^{\harm,0}_k\).  Decompose each term into its projected low part and high tail.  The low parts are included in \(P_k^{\low}\), and the high parts are included in \(P_k^{\trunc}\) and \(P_k^{\harm\text{-}\tail}\).  The projection norms and \(\rho^{-2}\) factor are absorbed into \(C_{\rho,N,H}\).
\end{proof}

\begin{theorem}[CKN extraction for the amended distance]
\label{thm:det-ckn-extraction}
There exists
\[
  \delta^{\can+}_*=\delta^{\can+}_*(\varepsilon_{\CKN},\rho,\lambda,N,H)>0
\]
such that
\[
  \delta^{\can+}_k<\delta^{\can+}_*\quad\Longrightarrow\quad \Phi_k(\rho)<\varepsilon_{\CKN}.
\]
Equivalently,
\[
  \Phi_k(\rho)\ge\varepsilon_{\CKN}\quad\Longrightarrow\quad \delta^{\can+}_k\ge\delta^{\can+}_*.
\]
\end{theorem}

\begin{proof}
If \(\delta^{\can+}_k<\delta\), then \(A_k^{\core}+E_k^{\core}<\delta\) and \(P_k^{\tail,+}<\delta\).  By \cref{lem:det-velocity-interp}, the velocity contribution to \(\Phi_k(\rho)\) is bounded by \(C'_{\rho,\lambda}\delta^{3/2}\).  By \cref{lem:det-pressure-control}, the pressure contribution is bounded by \(C_{\rho,N,H}\delta\).  Choose \(\delta=\delta^{\can+}_*\) sufficiently small so that
\[
  C'_{\rho,\lambda}\delta^{3/2}+C_{\rho,N,H}\delta<\varepsilon_{\CKN}.
\]
This proves the implication; the contrapositive gives the second statement.
\end{proof}

\begin{theorem}[Finite-chain counting for the amended canonical detector]
\label{thm:det-main-counting}
Let \(w_k\ge0\).  Assume the finite-window constants in \cref{thm:det-amended-upper,thm:det-lower-audit,thm:det-ckn-extraction} are valid for \(k=0,\ldots,K\).  With
\[
  \Bad_{\can}:=\{0\le k\le K:\Phi_k(\rho)\ge\varepsilon_{\CKN}\},
\]
one has
\[
  \sum_{k\in\Bad_{\can}}w_k
  \le
  C_{\can+,K}
  \sum_{k=0}^{K}w_k(L_k^{\AEF}+P_k^{\tail,+}+R_k^{\PFE,+}),
\]
where one may take
\[
  C_{\can+,K}=\frac{C^{\can+,\mathrm{up}}_{\max,K}}{c^{\can+}_{\min,K}\delta^{\can+}_*},
\]
with
\[
  C^{\can+,\mathrm{up}}_{\max,K}:=\max_{0\le k\le K}C^{\can+,\mathrm{up}}_{k,N,H},
  \qquad
  c^{\can+}_{\min,K}:=\min_{0\le k\le K}c^{\can+}_{k,N,H}.
\]
\end{theorem}

\begin{proof}
For \(k\in\Bad_{\can}\), \cref{thm:det-ckn-extraction} gives \(\delta^{\can+}_k\ge\delta^{\can+}_*\).  Hence
\[
  c^{\can+}_{\min,K}\delta^{\can+}_*\sum_{k\in\Bad_{\can}}w_k
  \le \sum_{k=0}^{K}w_k c^{\can+}_{k,N,H}\delta^{\can+}_k.
\]
The lower audit bound controls the right-hand side by \(\sum_k w_kM^{\can+}_{k,N,H}(D_k)\).  The upper realization then gives the stated estimate after division by \(c^{\can+}_{\min,K}\delta^{\can+}_*\).
\end{proof}

\begin{remark}[Comparison of the two counting theorems]
The standard-PDE counting theorem \cref{thm:weighted-counting} and the detector counting theorem \cref{thm:det-main-counting} have the same finite-chain form but different closing mechanisms.  The former closes through one-component compactness and uses \(C_{3,k}\) as the proved channel.  The latter closes because the amended detector contains CKN-compatible energy and pressure coordinates.  They should not be presented as identical theorems.
\end{remark}


\begin{thebibliography}{99}

\bibitem{Leray1934}
J.~Leray,
\newblock Sur le mouvement d'un liquide visqueux emplissant l'espace,
\newblock \emph{Acta Mathematica} \textbf{63} (1934), 193--248.
\newblock DOI: \url{https://doi.org/10.1007/BF02547354}.

\bibitem{Hopf1951}
E.~Hopf,
\newblock {\"U}ber die Anfangswertaufgabe f{\"u}r die hydrodynamischen Grundgleichungen. Erhard Schmidt zu seinem 75.~Geburtstag gewidmet,
\newblock \emph{Mathematische Nachrichten} \textbf{4} (1950/51), no.~1--6, 213--231.
\newblock DOI: \url{https://doi.org/10.1002/mana.3210040121}.

\bibitem{Scheffer1976}
V.~Scheffer,
\newblock Partial regularity of solutions to the Navier--Stokes equations,
\newblock \emph{Pacific Journal of Mathematics} \textbf{66} (1976), no.~2, 535--552.
\newblock DOI: \url{https://doi.org/10.2140/pjm.1976.66.535}.

\bibitem{Scheffer1977}
V.~Scheffer,
\newblock Hausdorff measure and the Navier--Stokes equations,
\newblock \emph{Communications in Mathematical Physics} \textbf{55} (1977), no.~2, 97--112.
\newblock DOI: \url{https://doi.org/10.1007/BF01626512}.

\bibitem{CKN1982}
L.~Caffarelli, R.~Kohn, and L.~Nirenberg,
\newblock Partial regularity of suitable weak solutions of the Navier--Stokes equations,
\newblock \emph{Communications on Pure and Applied Mathematics} \textbf{35} (1982), no.~6, 771--831.
\newblock DOI: \url{https://doi.org/10.1002/cpa.3160350604}.

\bibitem{SohrWahl1986}
H.~Sohr and W.~von Wahl,
\newblock On the regularity of the pressure of weak solutions of Navier--Stokes equations,
\newblock \emph{Archiv der Mathematik} \textbf{46} (1986), 428--439.
\newblock DOI: \url{https://doi.org/10.1007/BF01210782}.

\bibitem{Lin1998}
F.-H.~Lin,
\newblock A new proof of the Caffarelli--Kohn--Nirenberg theorem,
\newblock \emph{Communications on Pure and Applied Mathematics} \textbf{51} (1998), no.~3, 241--257.
\newblock DOI: \url{https://doi.org/10.1002/(SICI)1097-0312(199803)51:3<241::AID-CPA2>3.0.CO;2-A}.

\bibitem{ConstantinETiti1994}
P.~Constantin, W.~E, and E.~S. Titi,
\newblock Onsager's conjecture on the energy conservation for solutions of Euler's equation,
\newblock \emph{Communications in Mathematical Physics} \textbf{165} (1994), no.~1, 207--209.
\newblock DOI: \url{https://doi.org/10.1007/BF02099744}.

\bibitem{Eyink1994}
G.~L. Eyink,
\newblock Energy dissipation without viscosity in ideal hydrodynamics. I. Fourier analysis and local energy transfer,
\newblock \emph{Physica D} \textbf{78} (1994), no.~3--4, 222--240.
\newblock DOI: \url{https://doi.org/10.1016/0167-2789(94)90117-1}.

\bibitem{DuchonRobert2000}
J.~Duchon and R.~Robert,
\newblock Inertial energy dissipation for weak solutions of incompressible Euler and Navier--Stokes equations,
\newblock \emph{Nonlinearity} \textbf{13} (2000), no.~1, 249--255.
\newblock DOI: \url{https://doi.org/10.1088/0951-7715/13/1/312}.

\bibitem{SereginSverak2002}
G.~A. Seregin and V.~\v{S}ver\'ak,
\newblock Navier--Stokes equations with lower bounds on the pressure,
\newblock \emph{Archive for Rational Mechanics and Analysis} \textbf{163} (2002), no.~1, 65--86.
\newblock DOI: \url{https://doi.org/10.1007/s002050200199}.

\bibitem{ESS2003}
L.~Escauriaza, G.~Seregin, and V.~\v{S}ver\'ak,
\newblock $L_{3,\infty}$-solutions of Navier--Stokes equations and backward uniqueness,
\newblock \emph{Russian Mathematical Surveys} \textbf{58} (2003), no.~2, 211--250.
\newblock DOI: \url{https://doi.org/10.1070/RM2003v058n02ABEH000609}.

\bibitem{KukavicaZiane2006}
I.~Kukavica and M.~Ziane,
\newblock One component regularity for the Navier--Stokes equations,
\newblock \emph{Nonlinearity} \textbf{19} (2006), no.~2, 453--469.
\newblock DOI: \url{https://doi.org/10.1088/0951-7715/19/2/012}.

\bibitem{CaoTiti2011}
C.~Cao and E.~S. Titi,
\newblock Global regularity criterion for the 3D Navier--Stokes equations involving one entry of the velocity gradient tensor,
\newblock \emph{Archive for Rational Mechanics and Analysis} \textbf{202} (2011), no.~3, 919--932.
\newblock DOI: \url{https://doi.org/10.1007/s00205-011-0439-6}.

\bibitem{SereginLectureNotes}
G.~A. Seregin,
\newblock \emph{Lecture Notes on Regularity Theory for the Navier--Stokes Equations},
\newblock World Scientific Publishing Co. Pte. Ltd., Hackensack, NJ, 2015, x+258 pp.
\newblock DOI: \url{https://doi.org/10.1142/9314}.

\bibitem{CheminZhang2016}
J.-Y.~Chemin and P.~Zhang,
\newblock On the critical one component regularity for 3-D Navier--Stokes system,
\newblock \emph{Annales scientifiques de l'\'{E}cole Normale Sup\'{e}rieure, S\'{e}rie 4} \textbf{49} (2016), no.~1, 131--167.
\newblock DOI: \url{https://doi.org/10.24033/asens.2278}.

\bibitem{CheminZhangZhang2017}
J.-Y.~Chemin, P.~Zhang, and Z.~Zhang,
\newblock On the critical one component regularity for 3-D Navier--Stokes system: general case,
\newblock \emph{Archive for Rational Mechanics and Analysis} \textbf{224} (2017), no.~3, 871--905.
\newblock DOI: \url{https://doi.org/10.1007/s00205-017-1089-0}.

\bibitem{HanLeiLiZhao2019}
B.~Han, Z.~Lei, D.~Li, and N.~Zhao,
\newblock Sharp one component regularity for Navier--Stokes,
\newblock \emph{Archive for Rational Mechanics and Analysis} \textbf{231} (2019), no.~2, 939--970.
\newblock DOI: \url{https://doi.org/10.1007/s00205-018-1292-7}.

\bibitem{BarkerPrange2021}
T.~Barker and C.~Prange,
\newblock Quantitative regularity for the Navier--Stokes equations via spatial concentration,
\newblock \emph{Communications in Mathematical Physics} \textbf{385} (2021), no.~2, 717--792.
\newblock DOI: \url{https://doi.org/10.1007/s00220-021-04122-x}.

\bibitem{KangNguyen2023}
K.~Kang and D.~D. Nguyen,
\newblock Local regularity criteria in terms of one velocity component for the Navier--Stokes equations,
\newblock \emph{Journal of Mathematical Fluid Mechanics} \textbf{25} (2023), no.~1, article no.~10, 15 pp.
\newblock DOI: \url{https://doi.org/10.1007/s00021-022-00754-8}.

\bibitem{AlbrittonBarkerPrange2023}
D.~Albritton, T.~Barker, and C.~Prange,
\newblock Epsilon regularity for the Navier--Stokes equations via weak--strong uniqueness,
\newblock \emph{Journal of Mathematical Fluid Mechanics} \textbf{25} (2023), no.~3, article no.~49, 12 pp.
\newblock DOI: \url{https://doi.org/10.1007/s00021-023-00780-0}.

\bibitem{YuOneComponent2026}
R.~Yu,
\newblock Finite-Scale One-Component Regularity via Harmonic Pressure for the 3D Navier--Stokes Equations,
\newblock arXiv preprint arXiv:2606.08352 [math.AP], 2026.
\newblock DOI: \url{https://doi.org/10.48550/arXiv.2606.08352}.

\bibitem{YuStrict2026}
R.~Yu,
\newblock Strict 2.5D Shadows for One-Component Navier--Stokes Regularity,
\newblock arXiv preprint arXiv:2606.11720 [math.AP], 2026.
\newblock DOI: \url{https://doi.org/10.48550/arXiv.2606.11720}.

\bibitem{YuSchur2026}
R.~Yu,
\newblock Schur Visibility and Anti-Phantom Reduction in One-Component Navier--Stokes Degeneration,
\newblock arXiv preprint arXiv:2606.12267 [math.AP], 2026.
\newblock DOI: \url{https://doi.org/10.48550/arXiv.2606.12267}.

\bibitem{YuInvisible2026}
R.~Yu,
\newblock Invisible Defect Cascades for Navier--Stokes Regularity,
\newblock arXiv preprint arXiv:2606.12756 [math.AP], 2026.
\newblock DOI: \url{https://doi.org/10.48550/arXiv.2606.12756}.

\bibitem{YuCriticalLedgers2026}
R.~Yu,
\newblock Critical Ledgers and Scale-Defect Cascades for Navier--Stokes,
\newblock arXiv preprint arXiv:2606.13887 [math.AP], 2026.
\newblock DOI: \url{https://doi.org/10.48550/arXiv.2606.13887}.

\bibitem{YuSingularityAuditTransfer2026}
R.~Yu,
\newblock Finite-Window Singularity Audits and Local-to-Clean Defect Transfer for Navier--Stokes,
\newblock arXiv preprint arXiv:2606.15086 [math.AP], 2026.
\newblock DOI: \url{https://doi.org/10.48550/arXiv.2606.15086}.

\bibitem{YuComputationalAntiPhantom2026}
R.~Yu,
\newblock Finite-Window Computational Anti-Phantom Theorems for Scale-Critical Navier--Stokes Defects,
\newblock arXiv preprint arXiv:2606.15456 [math.AP], 2026.
\newblock DOI: \url{https://doi.org/10.48550/arXiv.2606.15456}.

\end{thebibliography}
\end{document}